\documentclass[letterpaper,11pt]{amsart}
\usepackage{amssymb, graphicx, color, mathtools, pinlabel, hyperref, tikz} 
\usepackage[font=footnotesize]{caption}

\newtheorem{thm}{Theorem}
\newtheorem*{thm*}{Theorem}
\newtheorem{lem}[thm]{Lemma}
\newtheorem{cor}[thm]{Corollary}
\newtheorem{prop}[thm]{Proposition}

\newtheorem{conj}{Conjecture}
\newtheorem{defn}[thm]{Definition}
\theoremstyle{remark}

\newtheorem*{rem}{Remark}

\usetikzlibrary{calc,decorations.markings}
\usetikzlibrary{shapes,snakes}

\theoremstyle{definition}
\newtheorem*{defn*}{Definition}

\newcommand\Z{\mathbb{Z}}

\newcommand{\C}{\mathcal{C}}

\newcommand\Mod{\operatorname{Mod}}

\title[Aysmptotic translation length]{ An upper bound on the asymptotic translation lengths on the curve graph and fibered faces }

\author{Hyungryul Baik}
\address{Department of Mathematical Sciences, KAIST,  
291 Daehak-ro, Yuseong-gu, Daejeon 34141, South Korea }
\email{hrbaik@kaist.ac.kr}

\author{Hyunshik Shin}
\address{%
        Department of Mathematics, University of Georgia,
		Athens, GA 30602
}
\email{hyunshik.shin@uga.edu}

\author{Chenxi Wu}
\address{Department of Mathematics
Rutgers University
Hill Center - Busch Campus
110 Frelinghuysen Road
Piscataway, NJ 08854-8019, USA }
\email{cwu@math.rutgers.edu}
\keywords{curve complex, pseudo-Anosov, asymptotic translation length, fibered face}

\begin{document}

\begin{abstract}
We study the asymptotic behavior of the asymptotic translation lengths on the curve complexes of pseudo-Anosov monodromies in a fibered cone of a fibered hyperbolic 3-manifold $M$ with $b_1(M) \geq 2$. For a sequence $(\Sigma_n, \psi_n)$ of fibers and monodromies in the fibered cone, we show that the asymptotic translation length on the curve complex is bounded above by $1/|\chi(\Sigma_n)|^{1+1/r}$ as long as their projections to the fibered face converge to a point in the interior, where $r$ is the dimension of the $\psi_n$-invariant homology of $\Sigma_n$ (which is independent of $n$). 
As a corollary, if $b_1(M) = 2$, the asymptotic translation length on the curve complex of such a sequence of primitive elements behaves like $1/|\chi(\Sigma_n)|^{2}$. Furthermore, together with a work of E. Hironaka, our theorem can be used to determine the asymptotic behavior of the minimal translation lengths of handlebody mapping class groups and the set of mapping classes with homological dilatation one. 
\end{abstract}

\maketitle

\section{Introduction}
The {\em curve graph} $\C(\Sigma)$ of a hyperbolic surface $\Sigma$ is the simplicial graph whose vertex set is the set of isotopy classes of essential simple closed curves in $\Sigma$, and there is an edge between two vertices if these two isotopy classes can be realized by disjoint curves. This is the 1-skeleton of the {\em curve complex} introduced by Harvey \cite{harvey1981boundary}. We define a metric $d_\C$ on the curve graph by setting all edge lengths as $1$. 
For a homeomorphism $\psi$ of $\Sigma$, the {\em asymptotic translation length} $l_\psi$ of $\psi$ 
on $\mathcal{C}(\Sigma)$  is defined by

$$l_\psi = \liminf_{j \to \infty} \dfrac{ d_{\mathcal{C}}(\alpha, \psi^j(\alpha))}{j}, $$
where $\alpha$ is an essential simple closed curve in $\Sigma$. Masur--Minsky \cite{MasurMinsky99} showed that the curve graph is $\delta$-hyperbolic, that this asymptotic translation length is independent on the choice of $\alpha$, and also that $l_\psi$ is non zero if and only of $\psi$ is pseudo-Anosov. Bowditch \cite{Bowditch08} proved further that $l_\psi$ is always a rational number and there is a uniform upper bound on the denominator depending only on $\Sigma$. 

There has been a few research on the relationship between $l_\psi$ the pseudo-Anosov stretch factor of $\psi$
\cite{GadreHironakaKentLeininger13, AougabTaylor15, Valdivia17}.
There are also papers on the asymptotic behavior of the minimal asymptotic translation length for different surfaces $\Sigma$ with increasing complexity \cite{GadreTsai11,Valdivia14}.

More precisely, let $\Mod(\Sigma)$ be the mapping class group of a surface $\Sigma$ where $\Sigma$ is obtained from a closed surface with finitely many points removed (such removed points are called punctures). For background on mapping class groups, we refer the readers to  \cite{FarbMargalit12}. For any $H \subset \Mod(\Sigma)$, let $L_{\C}(H)$ be the minimum of $l_\psi$ among pseudo-Anosov homeomorphisms $\psi$ with $[\psi] \in H$.
Then we have
$$ L_{\C}(\Mod(\Sigma)) \asymp \frac{1}{|\chi(\Sigma)|^2},$$
where both sides are considered as functions in $|\chi(\Sigma)|$, 
and we write $F(n) \asymp G(n)$ if there exists a constant $C>0$ so that $1/C \leq F(n)/G(n) \leq C$ for all $n$.

By a hyperbolic fibered 3-manifold, we mean a 3-manifold which fibers over the circle and admits a complete hyperbolic metric with finite volume. 
Let $(\Sigma_n, \psi_n)$ be a sequence of fibers and monodromies in the fibered cone of a hyperbolic fibered
3-manifold $M$. We say $(\Sigma_n, \psi_n)$ has \emph{small asymptotic translation length} if  
$l_{\psi_n}$ is bounded from below and above by constant multiples of $1/\chi(\Sigma_n)^2$,
where the constant depends only on the fibered cone. 
Recently, the second author and Eiko Kin \cite{KinShin17} produced a family of sequences with small asymptotic translation length. 
%Using their technique, one can always construct such a sequence from any fibration $(\Sigma, \psi)$ of a fibered hyperbolic 3-manifold $M$ with $b_1(M) \geq 2$. 
As an application, they showed that the sequences of minimal asymptotic translation lengths in 
hyperelliptic mapping class group and hyperelliptic handlebody group, respectively, of the closed surface $\Sigma_g$ 
have small asymptotic translation length, i.e., behaves like $1/g^2$. 
%the same holds for the cases of hyperelliptic mapping class group and hyperelliptic handlebody group.  

In this paper, we generalized the method and result in \cite{KinShin17} to sequences 
contained in a proper subcone of the fibered cone 
(i.e., a subcone whose closure is contained in the interior of the fibered cone with the exception of the origin $0$).

More precisely, our main result of this paper is as follows. 

\newtheorem*{thm:main}{Theorem \ref{thm:main}}
\begin{thm:main} 
 Let $M$ be a hyperbolic fibered 3-manifold, 
 and let $P$ be a proper subcone of a fibered cone.
 Let $L$ be a rational subspace of $H^1(M)$ whose intersection with $P$ is of dimension $r+1$. 
 Then for the pseudo-Anosov map $\psi_\alpha: \Sigma_\alpha \rightarrow \Sigma_\alpha$ induced by any primitive integral element $\alpha\in L\cap P$, the asymptotic translation length on the corresponding curve complex satisfies $$l_{\psi_\alpha}\lesssim \frac{1}{|\chi(\Sigma_\alpha)|^{1+1/r}},$$ i.e., 
there exists $C = C(P) > 0$ such that $l_{\psi_\alpha} \leq \frac{C}{|\chi(\Sigma_\alpha)|^{1+1/r}}$ for all $\alpha \in L \cap P$. 
\end{thm:main} 

We remark that there exists a sequence of primitive cohomology classes in the fibered cone
which projectively converges to the boundary of the fibered face and $l_{\psi_n} \asymp 1/\chi_n$,
where $\psi_n$ is the corresponding monodromy and $\chi_n$ is the Euler characteristic of the fiber
(see \cite[Section 5.4]{HironakaQuotient}, in particular, Proposition 5.6 for the description of the sequence, and see 
\cite[Section 4]{GadreHironakaKentLeininger13} for the behavior of asymptotic translation lengths on the curve complex). Theorem \ref{thm:main} implies that such behavior of the asymptotic translation length cannot be achieved by a sequence which projectively converges to an interior point of the fibered cone. 

When $r=1$, the upper bound in Theorem \ref{thm:main} is the best possible asymptotic upper bound due to the lower bound by \cite{GadreTsai11}. 
We conjecture that our upper bound is sharp for all $r \geq 1$.  

\begin{conj}\label{conj:sharp} 
For every $r \geq1$, there exists a fibered hyperbolic 3-manifold $M$ with the first Betti number greater than $r+1$, a proper subcone of one of its fibered cones $P$, a rational subspace $L$ of $H^1(M)$ of dimension $r+1$, and a sequence $\{\alpha_j\}$ of primitive elements in $L\cap P$ such that if $\psi_{\alpha_j}: \Sigma_{\alpha_j} \rightarrow \Sigma_{\alpha_j}$ is induced by $\alpha_j$ for all $j$, then $l_{\psi_{\alpha_j}}\gtrsim \frac{1}{{|\chi(\Sigma_{\alpha_j})|}^{1+1/r}}$.
\end{conj}

In an upcoming paper of the authors with Eiko Kin \cite{baik2019asymptotic}, we studied about the magic 3-manifold $N$ and we could construct a sequence $\{\alpha_j\}$ of primitive classes in a fibered cone of the magic 3-manifold such that $l_{\psi_{\alpha_j}}\gtrsim \frac{1}{{|\chi(\Sigma_{\alpha_j})|}^{3/2}}$. Since the rank of $H^1(N)$ is 3, this sequence provides a partial evidence to Conjecture \ref{conj:sharp} by verifying the conjecture in the case $r = 2$. 

Together with the result of \cite{GadreTsai11}, Theorem \ref{thm:main} implies the following consequence.

\begin{cor} \label{cor:2dim}
Let $\{ \alpha_n \}$ be a sequence of primitive cohomology classes in $H^1(M;\Z)$ such that
it is contained in the 2-dimensional proper subcone of a fibered cone. Let $(\Sigma_n, \psi_n)$ be the corresponding fiber and pseudo-Anosov monodromy. Then it has small asymptotic translation length, that is,
$$l_{\psi_n} \asymp \frac{1}{\chi(\Sigma_n)^2}.$$
\end{cor}

As an application of our result, we show that
the handlebody mapping class group $\mathcal{H}_g$ and the set $\Delta_g$ of pseudo-Anosov mapping classes 
whose homological dilatation is one admit a sequence in terms of $g$ with small asymptotic translation length. 
Hence we obtain the following consequences.

\newtheorem*{thm:handlebody}{Theorem 9}
\begin{thm:handlebody}[Corollary 1.2 of \cite{KinShin17}] 
$L_{\mathcal{C}}( \mathcal{H}_g ) \asymp \dfrac{1}{g^2}$.
\end{thm:handlebody}

\newtheorem*{thm:homologicalone}{Theorem 10}
\begin{thm:homologicalone}[See also \cite{Hironaka11}]
$L_{\mathcal{C}}( \Delta_g ) \asymp \dfrac{1}{g^2}$.
\end{thm:homologicalone}

In the case for the handlebody mapping class groups, it was obtained in \cite{KinShin17} by constructing a specific sequence satisfying the condition of their main result. 
We observe that it also follows directly from our main theorem and sequences constructed in \cite{Hironaka11}.

\subsection*{Outline}
In Section 2, we will review the concept and some properties of fibered cone and Teichm\"uller polynomials in \cite{Thurston1986} and \cite{McMullen00}. In Section 3, we will give a more explicit description of the surface $\Sigma_{\alpha}$ and map $\psi_{\alpha}$ based on the properties of fibered cone. The key observation is that $\Sigma_{\alpha}$ contains regions which locally resembles an abelian cover of $\Sigma$, on which $\psi_{\alpha}$ sends a fundamental domain to a number of adjacent fundamental domains. In Section 4, we will show that the behavior of $\psi_{\alpha}$ on this region is controlled by the shape of fibered cone, using McMullen's Teichm\"uller polynomials. In Section 5, we will prove the main theorem by choosing a simple closed curve in the middle of this region of $\Sigma_{\alpha}$ and then utilize the conclusion in Section 4.

\medskip
\section{Notation and background}

In this section, we will introduce notations and review prior results which would be used in our proof.

Let $\Sigma$ be a surface of finite genus with finitely many punctures and let $\psi$ be a pseudo-Anosov map. We use $M$ to denote the mapping torus $M=\Sigma\times[0,1]/(p,1)\sim(\psi(p),0)$, which is a surface bundle over the circle via the projection $\pi: (x, t)\mapsto t$ where $x\in \Sigma$, $t\in [0,1]$, and use $\mathcal{F}$ to denote the flow in $M$ determined by monodromy, i.e., the one determined by the vector field $\partial/\partial t$. It is evident that the time $1$ map of $\mathcal{F}$ sends $\Sigma$ to itself via $\psi$. Let $\alpha_\psi$ be the pullback of the standard generator of $H^1(S^1)$ under this fibration. In general, for any fibration of $M$ over the circle, if the pullback of the generator of $H^1(S^1)$ under this fibration is $\beta$, we denote this fibration as $\pi_\beta$, the monodromy map of this fibration as $\psi_\beta$ and the fibers as $\Sigma_\beta$ ($\Sigma_\beta$ has a single connected component iff $\beta$ is primitive). Lastly, we use $M_0\rightarrow M$ to denote the $\mathbb{Z}$-fold cover resulting from the pullback of the universal cover of $S^1$.

Thurston \cite{Thurston1986} showed that the elements in $H^1(M;\mathbb{Z})$ corresponding to fibrations over the circle are exactly those that lies in the union of finitely many disjoint open rational cone (cone determined by inequalities with rational coefficients), which are called {\em fibered cones}. Furthermore, McMullen \cite{McMullen00} showed that for each fibered cone, there is a Laurant polynomial which encodes the stretch factors of all the monodromies. Below we will review these results in a slightly more general setting where we replace $H^1(M)$ with a rational subspace $L$ of $H^1(M)$. The proofs of the theorems stated in the rest of this section is completely identical to the proofs for the case $L=H^1(M)$ in \cite{Thurston1986, McMullen00}.

Let $\widetilde{\Sigma}$ be a free abelian cover of $\Sigma$ such that $\psi$ is lifted to $\widetilde{\psi}$, with deck group $H$. The mapping torus $M_1$ of $\widetilde{\psi}$ is a free abelian cover of $M$. Let $\widetilde{M}$ be the fiber product (or smallest common abelian cover) of covers $M_1\rightarrow M$ and $M_0\rightarrow M$, which now has a deck group $\Lambda=H\oplus\mathbb{Z}$. Let $L'=\Lambda\otimes\mathbb{R}$, then $L'$ is a rational quotient of $H_1(M;\mathbb{R})$, hence its dual $L=(L')^*$ is a rational subspace of $H^1(M;\mathbb{R})$.

Let $\{h_1.\dots h_r\}$ be a basis of $H$, then there is a basis $\{H_1,\dots H_r,\Psi\}$ of $\Lambda$, such that $H_j$ preserves the leaves and when restricted to each leaf $H_j$ becomes $h_j$, and $\Psi$ is the time-1 map of the flow $\mathcal{F}$ corresponding to the monodromy map $\widetilde{\psi}$. Let $\{H'_i, \Psi'\}$ be the dual basis in $L\subset H^1(M;\mathbb{R})$. Then, it is evident that $\alpha_\psi$ has the coordinate $(0,\dots, 0, 1)$ under this basis.

Thurston \cite{Thurston1986} showed the following:

\begin{prop}\cite{Thurston1986}
\begin{itemize}
\item With the notation above, there is a unique rational cone $\mathcal{C}\subset L$ containing the point $\alpha_\psi$ such that all integer elements in the cone correspond to a fibration of $M$ over the circle, while integer elements on its boundary don't. This is called the \emph{fibered cone} containing $\psi_\alpha$.
\item For any integer class $\beta\in \mathcal{C}$, the leaves of the fibration over $S^1$ corresponding to $\beta$ are transverse to $\mathcal{F}$, and the monodromy $\psi_\beta$ is the first return map of the flow $\mathcal{F}$ on a leaf $\Sigma_\beta$.
\end{itemize}
\end{prop}

MuMullen \cite{McMullen00} introduced a polynomial invariant of the fibered cone $\mathcal{C}$ which will be used in the proof in the next section, so we will review its definition and some of its properties here.

\begin{defn}\cite{McMullen00}\label{McMpoly} Let $\tau$ be an invariant train track of $\psi$ which gets lifted to an invariant train track $\widetilde{\tau}$ on $\widetilde{\Sigma}$. The vector spaces generated by edges and vertices of $\widetilde{\tau}$ are finitely generated free $\mathbb{Z}[H]$ modules via the deck transformation. Let $P_E$ and $P_V$ be the map on these two vector spaces induced by $\widetilde{\psi}$. The Teichm\"uller polynomial for the fibered face containing $\alpha_\psi$ is now defined as:
  \[\Theta(u)={\det(uI-P_E)\over \det(uI-P_V)}\]
By identifying the group ring of $H$ with a polynomial ring, we can see $\Theta$ as a Laurant polynomial. Alternatively, by identifying $u$ with $\Psi$ we can see $\Theta$ as an element in $\Lambda$.
\end{defn}

A property of these polynomials we will use in this article is the following, which is a modified version of \cite[Theorem 6.1]{McMullen00}:

\begin{prop}\cite[Theorem 6.1]{McMullen00}\label{dual_cone} The dual cone of $\mathcal{C}$ is the cone on the Newton polytope of $\det(uI-P_E)$ centered at $u^n$, where $n$ is the highest degree of $u$. 
\end{prop}

Here, let $N$ be a polytope and $v$ a vertex of $N$, the cone on $N$ centered at $v$ is defined as the smallest cone centered at $0$ containing the set $\{v-x: x\in N\}$. The Newton polytope of $\theta=\sum_{g\in\Lambda}\theta_gg\in\mathbb{C}[\Lambda]$ is the convex hull of the set $\{\theta:h_\theta\not=0\}$, and by ``the dual cone of a set $S\subset V$'' we mean $\{x\in V^*:x(s)>0\text{ for all }s\in S\}$.

\begin{proof}
  We follow the proof of \cite[Theorem 6.1]{McMullen00}.
  
Let $t=(t_1,\dots, t_r)$ be the parameters of the Teichm\"uller polynomial corresponding to the basis $\{H_1, \dots, H_r\}$ of $H$. Let $E(t)$ be the leading eigenvalue of $P_E(t)$. Then, \cite[Theorem A.1(C)]{McMullen00} says that if a factor $F(u, t)$ of $uI-P_E(t)$ satisfies $F(t, E(t))=0$, then the cone spanned by the set $\{(s_1, \dots s_r, y)\in \mathbb{R}^r\times \mathbb{R}:y=\log E(e^{s_1},\dots e^{s_r})\}$ must be the cone which consists of the cohomology classes whose value on the Newton polytope of $\det(uI-P_E)$ takes maximum at $u^n$, where $n$ is the size of matrix $P_E$, which is also the fibered cone $\mathcal{C}$ due to \cite[Theorem 5.3]{McMullen00}. Hence, the cone in the Newton polytope $N(\det(uI-P_E))$ centered at $u^n$ is the dual cone of the fibered cone $\mathcal{C}$.  Now instead of letting $F=\Theta$ as in the proof of \cite[Theorem 6.1]{McMullen00}, we just let $F=uI-P_E(t)$, then the proposition is proved.
\end{proof}

\medskip
\section{Description of the fibers and monodromies}
\label{sec:fibermonodromy}

Let $\beta$ be a primitive integral element in $L$ with a coordinate of $(p_1, \dots , p_r, n)$ under the basis $(H_1',\dots H_r',\Psi')$. We will now give an explicit description of $\Sigma_\beta$ and $\psi_\beta$ which will be useful in later sections.

Let $\beta^\perp\subset\Lambda$ be the subgroup of $\Lambda$ consisting of elements which vanish when pairing with $\beta$. Let $\beta_1,\dots \beta_r$ be a basis for $\beta^\perp$, say $\beta_i=H_1^{x^i_1}\dots H_r^{x^i_r}\Psi^{y^i}$, and let $b_i=h_1^{x^i_1}\dots h_r^{x^i_r}\widetilde{\psi}^{y^i}$, for $i=1,\dots r$. 

The main goal of this section is to prove the following lemma:

\begin{lem} 
\label{lem:liftingmonodromy} 
Let $\beta$ be a primitive integral element in $L$. 
  With the notation above, we have
  \begin{enumerate}
  \item $\Sigma_\beta=\widetilde{\Sigma}/\langle b_1,\dots b_r\rangle$
  \item $\psi_\beta$ has a lift $\widetilde{\psi_\beta}$ on $\widetilde{\Sigma}$ such that $\widetilde{\psi_\beta}^n=\widetilde{\psi}$ for some $b\in\langle b_1,\dots b_r\rangle$.
  \end{enumerate}
\end{lem}

\begin{proof}
  For (1), consider $M_\beta=\widetilde{M}/\beta^\perp$. This is a $\mathbb{Z}$ fold cover of $M$ which preserves $\Sigma_\beta$ by construction, and the flow defined by $\mathcal{F}$ is a $\mathbb{R}$-action on $M_\beta$ whose orbits (i.e. the flow lines) intersect with $\Sigma_\beta$ exactly once. Hence we can identify $\Sigma_\beta$ with the set of these flow lines on $M_\beta$. Under the covering $\widetilde{M}\rightarrow M$, the surface $\Sigma$ is lifted to $\mathbb{Z}$-copies of $\widetilde{\Sigma}$. Pick one copy, quotient by $\langle b_1,\dots b_r\rangle$ and identify the points that lie on the same flow line of $\mathcal{F}$,
 we can see that two points in $\widetilde{\Sigma}$ are identified in this process iff they are related via an element in $b\in\langle b_1,\dots, b_r\rangle$. As the space of flow lines in $M_\beta$ one-one corresponds to elements in $\Sigma_\beta$, this proves (1).

For (2), consider the mapping torus of $\psi_\beta$ which is a $\beta^\perp$ cover of $M$. The time 1 flow along $\mathcal{F}$ restricted to $\Sigma_\beta$ is conjugate with $\widetilde{\psi}$ via the flow, and it is evident that during time $1$ the flow hits $\Sigma_\beta$ $n$ times, hence the first return map, which is a lift of $\psi_\beta$, satisfies that its $n$-th power is $\widetilde{\psi}$.
\end{proof}

\medskip
\section{The relationship between $\widetilde{\psi}$ and the fibered cone}

We keep the same notation as in the last two sections, let $D_0$ be a fundamental domain of the covering $\widetilde{\Sigma}\rightarrow\Sigma$. For every map $f: \widetilde{\Sigma}\rightarrow\widetilde{\Sigma}$, let $\Omega(f)$ be the convex hull of $\{x\in H: D_x\cap f(D_0)\not=\emptyset\}$ in $H\otimes\mathbb{R}$. We will show that one can use $\Omega(\widetilde{\psi}^p)$ to describe the dual of the fibered cone of the mapping torus $M$ in $L\subset H^1(M;\mathbb{R})$ as follows:

\begin{prop} \label{prop:dual_cone}
The set $\Omega=\cup_{p\in\mathbb{N}}(-\Omega(\widetilde{\psi}^p))\times\{p\}$ is contained in some $C$-neighborhood of $\mathcal{C}^*$, where $\mathcal{C}^*\subset L^*=(H\oplus\mathbb{Z})\otimes\mathbb{R}$ is the dual of the fibered cone containing $\alpha_\psi$. Recall from earlier that the dual of a cone is the set of elements in the dual space whose pairing with anything in the cone is positive, and the distance on $L^*$ can be chosen as any Euclidean norm.
\end{prop}

\begin{proof}
First we assume $r=1$. Let $N_1(p)$, $N_2(p)$ be two integer valued function such that $\Omega(\widetilde{\psi}^p)=[N_2(p),N_1(p)]$ and without loss of generality, we always assume that $N_2(p) \leq N_1(p)$. The proposition in this special case can be proved via the following two steps. \\

Step 1. {\em Reduce to the train-track maps:} Let $\tau$ be the invariant train-track of $\psi$. Let $\widetilde{\tau}$ and $\overline{\psi}$ be the $\mathbb{Z}$-fold cover of $\tau$ and the corresponding lift of the train-track map to this $\mathbb{Z}$-fold cover, respectively, and the lift $\overline{\psi}$ is chosen to be compatible with the splitting $L^*=(H\oplus \mathbb{Z})\otimes\mathbb{R}$. We further let $\tau_n=\widetilde{\tau}\cap D_n$ be the fundamental domains. Let $N'_1(p)=\inf\{n: \tau_n\cap \overline{\psi}^p(\tau_0)\not=\emptyset\}$, and $N'_2(p)=\sup\{n: \tau_n\cap \overline{\psi}^p(\tau_0)\not=\emptyset\}$. Then $N'_1-N_1$ and $N'_2-N_2$ are both uniformly bounded (actually $0 \leq N_i(p) - N_i'(p) \leq d$, where $d$ is the largest number such that $D_0$ and $D_d$ shares a boundary, because if the image of $D_0$ under $\widetilde{\psi}^p$ passes through two fundamental domains $D_{x}$ and $D_{x+d+1}$, it must pass through $D_{x+j}$ for some $1\leq j\leq d$ too and projects to something non-empty on $\tau_{x+j}$). \\

Step 2. {\em Relate $N'_i$ with the shape of the fibered cone:} Consider the incidence matrix $M_{\overline{\psi}}$ of $\overline{\psi}$ as a $\mathbb{Z}[t,t^{-1}]$-valued matrix of size $n\times n$, where $t$ is the symbol used to denote the different fundamental domains. In other words, if $h\in H$ is a basis then multiplication by $t$ is the same as action by $h$.  Let $F_p(x, t)=\det(M_{\overline{\psi}^p}-xI)$. Because $M_{\overline{\psi}^p}$ is just the $P_E$ in Definition \ref{McMpoly}, Proposition \ref{dual_cone} implies that the cone in the Newton polytope of $F_p$ centered at the vertex $x^n$ is the dual cone of the fibered cone of a $p$-fold cover of $M$ which the mapping torus of $\psi^p$. 

Identify $L^*$ with $\mathbb{R}^2$, such that $(1, 0)$ is the generator of $H$ while $(0, 1)$ is the generator of the direct summand $\mathbb{Z}$ in $L^*=(H\oplus\mathbb{Z})\otimes\mathbb{R}$. Suppose the dual cone $\mathcal{C}^*=\{(x, y): \alpha y<x<\beta y\}$ for some $\alpha<\beta$. Then, the dual cone of the mapping torus of $\psi^p$ must be contained in $\mathcal{C}_p^*=\{(x, y):p\alpha y<x<p\beta y\}$, because a fibered class in $\mathcal{C}$ can always be lifted to a fibered class in a fibered cone of the mapping torus of $\psi^p$ via the covering map.

Now we calculate the cone in the Newton polytope of $F_p$ centered at the point $(0, n)$. Let the maximal and minimal $t$-degree of the coefficient of $x^{n-k}$ in $F_p$ be $a_k(p)$ and $b_k(p)$, respectively. Then, the definition of Newton polytope implies that points $(-a_k(p), k)$ and $(-b_k(p), k)$ are all in $\mathcal{C}_p^*$. Hence $\alpha\leq-\max_k\{{a_k(p)\over kp}\}$, $-\min_k\{{b_k(p)\over kp}\}\leq \beta$.

Now let $p_0$ be such that $M_{\overline{\psi}^{p_0}}$ has no non-zero entry. Let $C$ be the difference of $t$-degrees of entries in $M_{\overline{\psi}^{p_0}}$. Then, as long as $p'>p_0$, because $M_{\overline{\psi}^{p'}}=M_{\overline{\psi}^{p_0}}M_{\overline{\psi}^{p'-p_0}}$, the difference in the highest as well as lowest $t$-degree of entries on for each column of $M_{\overline{\psi}^{p'}}$ is bounded within $C$ from each other, because they are linear combinations of a column of $M_{\overline{\psi}^{p'-p_0}}$ with coefficients being entries of $M_{\overline{\psi}^{p_0}}$. Similarly, because $M_{\overline{\psi}^{p'}}=M_{\overline{\psi}^{p'-p_0}}M_{\overline{\psi}^{p_0}}$, the highest as well as lowest $t$-degree of entries on each row of  $M_{\overline{\psi}^{p'}}$ is bounded within $C$ from each other. By definition of $M_{\overline{\psi}^{p'}}$, the highest and lowest $t$-degree of all entries of  $M_{\overline{\psi}^{p'}}$ are $N_1'(p')$ and $N_2'(p')$ respectively. Hence, the highest $t$-degree of $Tr(M_{\overline{\psi}^{p'}})$, which is $a_1(p')$, is no less than $N_1'(p)-2C$, while the lowest $t$-degree of this trace, which is $b_1(p')$, is no more than $N_2'(p')+2C$. Now we have $\alpha\leq -{N_1'(p')-2C\over p'}$, $-{N_2'(p')+2C\over p}\leq \beta$. Hence, $\Omega=\cup_{p\in\mathbb{N}}[-N_1(p), -N_2(p)]\times\{p\}$ is indeed contained in a uniform neighborhood of the cone $\mathcal{C}^*$.

Now let us consider the general case (i.e., $r$ is no longer assumed to be $1$). Recall that the dual of the fibered cone is a rational convex polytope.  
Let $F_1, \ldots, F_\ell$ be its faces. For each $F_i$, choose a 1-dimensional subspace $L_i$ of $H$ which is perpendicular to $F_i$. Then the case of $r=1$ above shows that there exists $C >0$ which does not depend on $p$ so that the projection of the set $\cup_p (-\Omega(\widetilde{\psi}^p)) \times p$ onto $(L_i\oplus\mathbb{Z})\otimes\mathbb{R}$ is in the some $C$-neighborhood of the projection of the dual of the fibered cone onto $L_i$ for each $i$. This, by geometry, implies that indeed the set $\cup_p (-\Omega(\widetilde{\psi}^p)) \times p$ is contained in some $C$-neighborhood of the dual of the fibered cone.
\end{proof}

\begin{rem}
In fact, one can strengthen the statement of the proposition to conclude that $\cup_p \Omega(\widetilde{\psi}^p) \times p$ and the dual of the fibered cone are at most $C$ Hausdorff distance apart. 
\end{rem}

\medskip
\section{Small asymptotic translation lengths in the fibered cone} 

In this section we prove the main theorem.

\begin{thm} \label{thm:main}
 Let $M$ be a hyperbolic fibered 3-manifold, 
 and let $P$ be a proper subcone of a fibered cone.
 Let $L$ be a rational subspace of $H^1(M)$ whose intersection with $P$ is of dimension $r+1$. 
 Then for the pseudo-Anosov map $\psi_\alpha: \Sigma_\alpha \rightarrow \Sigma_\alpha$ induced by any primitive integral element $\alpha\in L\cap P$, the asymptotic translation length on the corresponding curve complex satisfies $$l_{\psi_\alpha}\lesssim \frac{1}{|\chi(\Sigma_\alpha)|^{1+1/r}},$$ i.e., 
there exists $C = C(P) > 0$ such that $l_{\psi_\alpha} \leq \frac{C}{|\chi(\Sigma_\alpha)|^{1+1/r}}$ for all $\alpha \in L \cap P$. 
\end{thm}

\vspace{.5em}
\begin{proof}[Proof of Theorem \ref{thm:main}]
We use the same notation as in Section 2 and 3. Let $\Gamma$ be the free abelian group generated by $b_1, \ldots, b_r$, and let  $\zeta: \Gamma \to \mathbb{Z}^r$ be
the homomorphism sending $h_1^{x_1}\dots h_r^{x_r}\widetilde{\psi}^y$ to $(x_1, \ldots, x_r)$. 
Furthermore, we give $L^*$ a Euclidean metric such that the basis $\{H_1, \dots, H_r,\Psi\}$ is orthonormal. 

We begin by proving the following:
\begin{lem}	\label{lem:bilip}
There is some $\epsilon>0$ so that for any primitive integer element $\alpha\in L\cap P$, any $b=h_1^{x_1}\dots h_r^{x_r}\widetilde{\psi}^y\in \alpha^\perp \subset\Gamma$, any $p\in\Omega(b)$, where $\Omega(b)$ is defined as the convex hull of $\{x\in H: D_x\cap f(D_0)\not=\emptyset\}$, we have $\epsilon \cdot  d(0,\zeta(b))\leq d(0,p)\leq {d(0,\zeta(b))\over\epsilon}$ when $d(0,\zeta(b))$ is sufficiently large.
\end{lem}

\begin{proof}
  Because $\alpha$ is in a proper subcone of the fibered cone, the point $(x_1,\dots, x_r, y)$ corresponding to $b$ lies outside a cone that contains the dual cone as a proper subcone, hence, in combination of Proposition \ref{prop:dual_cone}, we have that the distance from $x=(x_1, \dots, x_r)$ to $-\Omega(\widetilde{\psi}^y)$ is bounded from below by some function of the form $Ay-C$, where $A, C>0$ are some real numbers. Also, Proposition \ref{prop:dual_cone} implies that $\Omega(\widetilde{\psi}^y)$ (as well as $-\Omega(\widetilde{\psi}^y)$)  is contained in some disc centered at $0$ and with a radius bounded from above by some function $A'y+C'$, for some $A', C'>0$.

  Now by the definition of $\Omega(b)$, we have $\Omega(b)=x+\Omega(\widetilde{\psi}^y)$. Hence, if $p\in \Omega(b)$, $x-p\in -\Omega(\widetilde{\psi}^y)$. As a consequence, $d(0, p)=d(x, x-p)\geq Ay-C$, and $d(0, p)=d(x, x-p)\geq d(x, 0)-d(0, p)=d(x, 0)-A'y-C'$. So
\begin{align*}  
  d(0, p) & \geq\max\{Ay-C, d(x, 0)-A'y-C'\}\\
  & \geq \begin{cases} Ay-C & d(x, 0)\leq (A'+1)y+C' \\ d(x, 0)-A'y-C'  & d(x, 0)> (A'+1)y+C'\end{cases}
 \end{align*}                                                                        
 Which is $\geq \epsilon \cdot d(0, x)$ when $d(0, x)=d(0, \zeta(b))$ is large enough.

 The second part of the inequality is because $d(0, p)\leq d(0, x)+d(0, x-p)\leq d(0, x)+A'y+C'$, and $y'$ is bounded from above by a multiple of $d(x, 0)$ because $b$ lies outside the dual cone of the fibered cone.
\end{proof}

Let $\mathcal{L}$ be the image of $\zeta$. Then we have:

\begin{lem} \label{lem:covolume}
The covolume of $\mathcal{L}$ is $\gtrsim n$. Here $a=(p_1,\dots, p_r, n)$ under the basis $\{H_1,\dots, H_r, \Psi\}$, or, in other words (cf. Lemma 3), $n$ is smallest number such that $\widetilde{\psi_\alpha}^n=b\widetilde{\psi}$ where $b\in \langle b_1,\dots, b_r\rangle$. 
\end{lem}
\begin{proof}
The sublattice $\alpha^\perp$ together with any vector not on the plane spanned by $\alpha^\perp$ forms a sublattice in $\mathbb{Z}^{r+1}$ and hence it has a nonzero integer as its covolume. Since $\alpha$ is primitive, there must be some other primitive integer element $\alpha'$ so that its inner product with $\alpha$ is $1$, which means that $\alpha'$ is distance $1/|\alpha|$ from the plane spanned by $\alpha^\perp$. Therefore the covolume of $\alpha^\perp$ is greater than $|\alpha|=\sqrt{n^2+\sum_i p_i^2}\geq n$. Note that $\mathcal{L}$ is the orthogonal projection of $\alpha^\perp$ and the angle between the planes they span is bounded because $\alpha$ is in a convex cone. Hence its covolume is $\gtrsim n$.
\end{proof}

\begin{figure}[t]
\begin{tikzpicture}
\draw [fill=black] (0,0) circle[radius=0.5pt];
\draw [fill=black] (0.5,0) circle[radius=0.5pt];
\draw [fill=black] (1,0) circle[radius=0.5pt];
\draw [fill=black] (-0.5,0) circle[radius=0.5pt];
\draw [fill=black] (1.5,0) circle[radius=0.5pt];
\draw [fill=black] (-1.5,0) circle[radius=0.5pt];
\draw [fill=black] (-1,0) circle[radius=0.5pt];
\draw [fill=black] (0,5.5) circle[radius=0.5pt];
\draw [fill=black] (0.5,5.5) circle[radius=0.5pt];
\draw [fill=black] (1,5.5) circle[radius=0.5pt];
\draw [fill=black] (-0.5,5.5) circle[radius=0.5pt];
\draw [fill=black] (-1,5.5) circle[radius=0.5pt];
\draw (0.1,-0.4) rectangle (0.9,0.4);
\draw (0.2,-0.8) rectangle (1.8,0.8);
\draw (-0.1,-0.4) rectangle (-0.9,0.4);
\draw (-0.2,-0.8) rectangle (-1.8,0.8);
\draw (0.3,-1.2) rectangle (2.7,1.2);
\draw (-0.3,-1.2) rectangle (-2.7,1.2);
\draw (-3.5,2) rectangle (3.5,9);
\draw [densely dotted, thick](0,0) circle[radius=1.5];
\draw [densely dotted, thick](-0.5,-2)--(0.5,2);
\draw [densely dotted, thick](-0.5,2)--(0.5,-2);
\draw [fill=black] (0,1) circle[radius=0.5pt];
\node at (0,-0.22) {$0$};
\node at (0,0.8) {$y$};
\end{tikzpicture}
\caption{An illustration of the second case of the proof of Lemma \ref{lem:distance} when $r=2$. The plane is a two dimensional quotient of $H_1(\Sigma)_{inv}$, and the dual cone of the fibered cone is drawn as a regular square pyramid. The squares are $\Omega(b)$ (which by Proposition \ref{prop:dual_cone} are asymptotically close to the negative of the slices of the dual cone, which are convex polytopes, in this case, squares), and the dotted circle is of radius $\epsilon\eta(n)$. One can now see clearly that the angle of the dotted sector is bounded from below by a multiple of $\epsilon$, because the squares on the horizontal line increases linearly as they are from slices of the dual cone at arithmetically increasing heights.}
\label{figure:cone}
\end{figure}

\begin{lem}	\label{lem:distance}
For $n>>0$, there is some $y$ so that the distance from $y$ to $\{0\}\cup\bigcup_{b\in\langle b_1,\dots b_r\rangle}\Omega(b)$ is $\gtrsim n^{1/r}$.
\end{lem}

\begin{proof}
Let $n\rightarrow\infty$. By Lemma \ref{lem:covolume}, either the lattice $\mathcal{L}$ does not degenerate, i.e., its systole (minimal distance between lattice points) is $\sim n^{1/r}$ or it degenerates, i.e., its systole grows slower than $n^{1/r}$. In the former case, due to Lemma \ref{lem:bilip}, one can find a ball centered at $0$ with radius $\sim n^{1/r}$ that does not intersects with the $\Omega(b)$, and $y$ can be found inside this ball. In the second case, the degeneration means that there is some sublattice $\mathcal{L}'$ with small covolume and that the cosets of $\mathcal{L}'$ are distance $\eta(n)$ apart where $\eta$ grows faster than $n^{1/r}$. Then $y$ can be found as a point that is $Cn^{1/r}$-away from $0$ and almost orthogonal to the subspace spanned by $\mathcal{L}'$ yet less than $\epsilon \eta(n)$ away from $0$ for some $\epsilon<<1$ as in Lemma \ref{lem:bilip} (See Figure \ref{figure:cone}). 
\end{proof}

Let $\gamma$ be a simple closed curve in $D_0$ and let $\gamma'$ a simple closed curve in $D_y$. By Lemma \ref{lem:distance} and Proposition \ref{prop:dual_cone}, we know that there is some $C>0$ so that $\widetilde{\psi}^{Cn^{1/r}}(\gamma')$ is disjoint from $\gamma$ as well as $b(\gamma)$ for every $b\in\langle b_1,\dots b_r\rangle$. By Lemma \ref{lem:liftingmonodromy} in Section \ref{sec:fibermonodromy}, this means that both $\gamma'$ and $\psi_{\alpha}^{n\cdot Cn^{1/r}}(\gamma')$ are distance $1$ from $\gamma$, hence the asymptotic translation length $l_\alpha$ is bonded above by ${2\over Cn^{1+1/r}}\sim |\chi(\Sigma_\alpha)|^{-1-1/r}$.  
\end{proof}

\medskip
\section{Applications} 

In this section we use our main theorem to determine the asymptotes of minimal asymptotic translation lengths of some subgroup/subset of $\Mod(\Sigma_g)$.

Let $\mathcal{H}_g$ be the handlebody mapping class group of the closed surface $\sigma_g$ of genus $g$. Then we
obtain a different proof of the following result of \cite[Corollary 1.2]{KinShin17}.
 
\begin{thm}\label{thm:handlebody}
$L_{\mathcal{C}}( \mathcal{H}_g ) \asymp \dfrac{1}{g^2}$.
\end{thm} 
\begin{proof} This is a direct consequence of Corollary \ref{cor:2dim} and 
the proof of Theorem 1.2 of \cite{Hironaka11}. More precisely, we can choose a sequence contained in a 
rational subspace of 2-dimensional fibered cone as in Theorem 4.1 of \cite{Hironaka11}. Then our corollary gives an asymptotic bound of $\dfrac{1}{g^2}$.  
\end{proof} 

Let $\Delta_g$ be the subset of $\Mod(\Sigma_g)$ consisting of
pseudo-Anosov mapping classes whose homological dilatation is
one. Then we have 

\begin{thm} \label{thm:homologicalone}
$L_{\mathcal{C}}( \Delta_g ) \asymp \dfrac{1}{g^2}$.
\end{thm} 
\begin{proof} This follows from Corollary \ref{cor:2dim} and 
the proof of Theorem 1.3 of \cite{Hironaka11}. The construction of the sequence that realizes the asymptotic upper bound is done in Lemma 5.1 of \cite{Hironaka11}, which is contained in a rational subspace of 2-dimensional fibered cone. Hence our corollary gives an asymptotic bound of $\dfrac{1}{g^2}$.  
\end{proof}

\medskip
\section{Acknowledgements}

\noindent We thank Mladen Bestvina, Bal\'azs Strenner, Eiko Kin and Ki Hyoung Ko for helpful discussions. We especially thank Bal\'azs Strenner for pointing out an error in Lemma \ref{lem:liftingmonodromy} in the earlier draft. The first author was partially supported by Samsung Science \& Technology Foundation grant no. SSTF-BA1702-01.
The second author was supported by Basic Science Research Program
through the National Research Foundation of Korea(NRF) funded by the Ministry of Education (NRF-2017R1D1A1B03035017).
The third author appreciates KAIST for hospitality, during which time this work was carried out.

\vspace{1em}

\bibliographystyle{alpha} 
\bibliography{asymtrans}

\end{document}